\documentclass[12pt]{amsart}

\usepackage{amssymb}
\usepackage{amsthm}
\usepackage{amsmath}
\usepackage{fontenc}
\usepackage{inputenc}
\usepackage{enumitem}
\usepackage{hyperref}
\usepackage{times}
\usepackage{graphicx}
\usepackage{multirow}
\usepackage{xcolor}
\usepackage{colortbl}
\usepackage{color}
\usepackage{array}
\usepackage{wasysym}
\usepackage{tikz}
\usetikzlibrary{scopes,arrows,decorations.pathmorphing,backgrounds,positioning,fit,petri,shapes,calc}

\makeatletter
\@namedef{subjclassname@2010}{%

  \textup{2010} Mathematics Subject Classification}

\makeatother

\def\real{\hbox{\rm\setbox1=\hbox{I}\copy1\kern-.45\wd1 R}}
\def\natural{\hbox{\rm\setbox1=\hbox{I}\copy1\kern-.45\wd1 N}}

\newtheorem{theorem}{Theorem}[section] 
\newtheorem{lemma}[theorem]{Lemma}     
\newtheorem{corollary}[theorem]{Corollary}

\title{Mixing Sets for Non-mixing Transformations}

\author[T. M. Adams]{Terrence M. Adams}

\address{U.S. Government\\9800 Savage Rd.\\Ft. Meade, MD 20755}

\email{terry@ganita.org}

\date{\today}

\begin{document}

\begin{abstract} 
For different classes of measure preserving transformations, 
we investigate collections of sets that exhibit the property 
of lightly mixing.  Lightly mixing is a stronger property 
than topological mixing, and requires that a $\liminf$ 
is positive.  In particular, we give a straightforward proof that 
any mildly mixing transformation $T$ 
has a dense algebra $\mathcal{C}$ such that $T$ is lightly mixing 
on $\mathcal{C}$.  Also, we provide a hierarchy for the properties 
of lightly mixing, sweeping out and uniform sweeping out 
for dense collections, and dense algebras of sets. 
As a result, stronger mixing realizations are given for several types 
of transformations than those given by previous extensions 
of the Jewett-Krieger Theorem. 
\end{abstract}


\maketitle

\section{Introduction}
The notion of lightly mixing was introduced as intermixing by Peter Walters in \cite{Walt72}. 
In \cite{BCQ74}, an alternate formulation of lightly mixing, called sequence mixing, 
was introduced, and it was observed that sequence mixing is equivalent to lightly mixing. 
Also, a transformation was constructed in \cite{BCQ74} which is lightly mixing, 
but not partially mixing.  This is accomplished by taking the countable product 
of a non-mixing, partially mixing transformation. 

It is easy to verify that any finite product of partially mixing transformations will still 
be partially mixing.  This follows since the usual definition of partial mixing extends 
from a dense collection of sets to the sigma algebra.  Other properties such as strong mixing, 
mild mixing and weak mixing also extend from a dense collection, when the appropriate 
definition is taken.  However, we will give general results showing that the usual definition 
of lightly mixing and sequence mixing do not extend from a dense collection 
to the generated sigma algebra. 

In section \ref{prelim}, we define the notion of lightly mixing on a collection of sets, 
and in section \ref{lightsec}, 
we prove the general result that any measure preserving transformation is lightly mixing 
on a dense collection of sets.  If a transformation is rigid, then it cannot be lightly mixing 
on a collection which includes a set and its complement.  Thus, no rigid transformation 
is lightly mixing on a dense algebra.  However, we show that any invertible measure 
preserving transformation with a mildly mixing factor is lightly mixing on a dense algebra. 

We also consider the property of sequence mixing which we call sweeping out.  
We show that sweeping out on a dense collection implies lightly mixing on a dense algebra. 
However, the converse does not hold.  In particular, the well known Chacon's transformation 
\footnote{The Chacon-3 transformation can be defined by $B_{n+1} = B_n B_n 1 B_n$ and $B_0 = (0)$.} 
does not sweep out on a dense collection.  Thus, since Chacon's transformation is mildly mixing, 
it is lightly mixing on a dense algebra.  
We give several other examples exhibiting the rich behavior that is possible. 

\section{Preliminaries}
\label{prelim}
A transformation on a separable probability space $(X, \mathcal{B}, \mu)$ is a measurable and invertible 
map $T:X \to X$. 
A transformation $T$ is measure preserving if for each set $A\in \mathcal{B}$, $\mu (T^{-1}A)=\mu (A)$. 
$T$ is ergodic if $TA=A$ implies $\mu(A)$ equals 0 or 1. 

Given a collection $\mathcal{C}$ of measurable sets, $T$ is lightly mixing on $\mathcal{C}$, 
if for each $A$ and $B\in \mathcal{C}$ having positive measure,
\[
\liminf_{n\to \infty} \mu (T^n A\cap B) > 0. 
\]
$T$ is sweeping out on $\mathcal{C}$, 
if for each $A\in \mathcal{C}$ with positive measure 
and each infinite sequence $k_1, k_2, \ldots $, 
\[
\mu (\bigcup_{i=1}^{\infty} T^{k_i} A) = 1.
\]
Equivalently, $T$ is sweeping out on $\mathcal{C}$, 
if for each $A \in \mathcal{C}$ and $B\in \mathcal{B}$, both with positive measure, 
\[
\liminf_{n\to \infty} \mu (T^n A\cap B) > 0. 
\]
$T$ is lightly mixing on $\mathcal{B}$ if and only if $T$ is sweeping out on $\mathcal{B}$.  
In this case, we will just say that $T$ is lightly mixing (LM).  Transformation $T$ is partially mixing, 
if there exists $\alpha > 0$ such that for all sets $A$ and $B$, 
\[
\liminf_{n\to \infty} \mu (T^nA\cap B) > \alpha \mu (A) \mu (B). 
\]

Consider $\mathcal{B}$ as a separable metric space with the distance 
$d(A,B) = \mu (A\triangle B)$.  
We will say $T$ is {\bf lightly mixing on a dense collection} (LMDC), if there exists a dense subset 
$\mathcal{D} \subset \mathcal{B}$ such that $T$ is lightly mixing on $\mathcal{D}$.  
Similarly, we define {\bf lightly mixing on a dense algebra} (LMDA), 
{\bf sweeping out on a dense collection} (SODC), and 
{\bf sweeping out on a dense algebra} (DODA).  Clearly, LMDA implies LMDC and SODA implies SODC. 
In Corollary \ref{sodc_lmda}, we show that SODC implies LMDA.  Thus, we may list the properties from strongest 
to weakest as: SODA, SODC, LMDA, LMDC. 

While none of the properties LMDC, LMDA, SODC or SODA are sufficient to imply 
$T$ is lightly mixing, there are conditions which may be used to verify $T$ is lightly mixing.  
For example, if for all sets $A$ and $B$ of positive measure there exists an integer $N$ 
such that for $n\geq N$, $\mu (T^n A\cap B) > 0$, then $T$ is lightly mixing. 
(If the $\liminf$ was 0, then wait for the intersection to be less than ${\mu(B)}/{2^n}$ 
and subtract that part from $B$.)  If $T$ is ergodic and for each set $A$ of positive measure 
there exists $N$ such that for $n\geq N$, $\mu(T^n A\cap A) > 0$, then $T$ is lightly mixing. 
These properties imply that if $T$ is lightly mixing for a measure $\nu$, then $T$ is 
also lightly mixing for an equivalent measure $\mu$.  If $T$ is lightly mixing for a measure $\nu$, 
then $T$ has no weakly wandering set of positive measure.  (Set $A$ is weakly wandering 
for $T$ if there exists an infinite sequence $n_k$ such that $T^{n_k}A\cap T^{n_{k+1}}A = \emptyset$ 
for $k\in \natural$.)  Therefore, $T$ has a unique equivalent invariant measure $\mu$.  
See \cite{Fri70} for a discussion on existence of equivalent invariant measures.  Thus the lightly mixing 
property persists when we change to an equivalent invariant measure. 

The lightly mixing property we discuss here may seem to be reminiscent of the property 
of topological mixing.  A transformation $T$ is topologically mixing if for any two open sets 
$U$ and $V$, there exists $N$ such that for $n\geq N$, $\mu(T^n U\cap V)>0$.  
An extension of the Jewett-Krieger Theorem in \cite{Lehrer87} proves that any ergodic measure preserving 
transformation has a topological model which is a topologically mixing homeomorphism on a compact 
metric space.  In particular, topological mixing is not an isomorphism invariant.  
Yet, the properties LMDC, LMDA, SODC and SODA are all isomorphic invariants.  
Thus, there are transformations which are topologically mixing, but not LMDA.  
Also, example 1 in section \ref{exs} of this paper is LMDA, but not topologically mixing.  
We include the proof that example 1 is LMDA and see \cite{Pet89} for a proof that 
example 1 is not topologically mixing. 

\section{Lightly Mixing on a Dense Collection}
\label{lightsec}
In Theorem \ref{thm1}, we prove that any ergodic measure preserving system on a separable 
probability space is lightly mixing on a dense collection of sets.  
The ergodicity condition is dropped in Theorem \ref{thm2}. 
\begin{theorem} 
\label{thm1}
Let $T$ be an ergodic measure preserving transformation on a separable probability space 
$(X, \mathcal{B}, \mu)$.  There exists a dense collection $\mathcal{D} \subset \mathcal{B}$ such that 
\[
\liminf_{n\to \infty} \mu(A\cap T^n B) > 0 
\]
for all $A, B\in \mathcal{D}$. 
\end{theorem}
First we state and prove the following lemma which is used to construct the sets in $\mathcal{D}$. 
\begin{lemma}
\label{lem1}
Let $T$ be measure preserving on $(X, \mathcal{B}, \mu)$.  
Suppose for sets $A$ and $B$ there exist a set $\bar{A}$ and natural number $h$ such that 
$\mu ( \bigcup_{i=0}^{h-1} T^{-i} B ) = 1$ and 
$\bigcup_{i=0}^{h-1} T^i \bar{A}$ is a disjoint union contained inside $A$.  Then for all $n\in \natural$, 
\[
\mu(A\cap T^n B) \geq \mu (\bar{A}) .
\]
\end{lemma}

\proof[{\it Proof of Lemma \ref{lem1}}] 
\begin{eqnarray*}
\mu (\bar{A}) &=& \mu (\bar{A} \cap T^n (\bigcup_{i=0}^{h-1} T^{-i} B)) \\ 
&=& \mu (\bigcup_{i=0}^{h-1} (\bar{A} \cap T^{-i} T^n B)) \leq \sum_{i=0}^{h-1} \mu(\bar{A} \cap T^{-i} T^n B) \\
&=& \sum_{i=0}^{h-1} \mu(T^i \bar{A} \cap T^n B) \leq \mu (A\cap T^n B) .\ \Box 
\end{eqnarray*} 

\proof[Proof of Theorem \ref{thm1}] 
For each natural number $h$, let $B_h$ be the base of a Rohlin tower of height $h$ such that 
$E_h = X\setminus (\bigcup_{i=0}^{h-1} T^i B_h)$ satisfies $\mu (E_h) < 1 / h$.  
Choose set $\bar{A_h} \subset B_h$ such that $\mu (\bar{A_h}) = (1/h)\mu(B_h)$ and define 
$A_h = \bigcup_{i=0}^{h-1} T^i \bar{A_h}$.  
Define $D_h = B_h \cup E_h \cup A_h$ for $h\in \natural$.  By Lemma \ref{lem1}, for $h\leq k$ we have 
\[
\mu(D_k \cap T^n D_h) \geq \mu(A_k \cap T^n (B_h \cup E_h)) \geq \mu(\bar{A_k}) 
\]
for $h\in \natural$.  Therefore $T$ is lightly mixing on the collection 
\[
\mathcal{D} = \{ D\in \mathcal{B}: \mbox{there exists $h\in \natural$ such that $D_h \subset D$}\} .\ \Box 
\]

In the next theorem, we are able to drop the ergodicity condition by looking at the ergodic components of $T$. 

\begin{theorem}
\label{thm2} 
If $T$ is a measure preserving transformation on a separable probability space, then $T$ is lightly mixing 
on a dense collection. 
\end{theorem} 
\proof[Proof] 
Either $T$ has invariant sets with arbitrarily small measure, or $T$ has an 
ergodic component with positive measure.  Suppose $T$ has disjoint invariant sets $E_m$ 
of positive measure such that $\sum_{m = 1}^{\infty} \mu(E_m) < \infty$.  
Then $T$ is lightly mixing on the dense collection 
\[
\mathcal{D} = \{ D\in \mathcal{B}: \mbox{there exists $m\in \natural$ such that $\bigcup_{i=m}^{\infty} E_i \subset D$}\} . 
\]
Otherwise, if $T$ has an ergodic component of positive measure, then $T$ is lightly mixing 
on a dense collection by Theorem \ref{thm1}. $\Box$ 

\section{Lightly Mixing on a Dense Algebra} 
It is known that the class of rigid, weakly mixing transformations is residual in the space of Lebesgue measure 
preserving transformations on the unit interval, endowed with the weak-* topology.  
Since no rigid transformaition is lightly mixing on a dense algebra, there is a residual class of transformations 
which are not lightly mixing on a dense algebra. 

A transformation is mildly mixing (MM) if it has no rigid factor. 
The notion of mildly mixing was originally introduced by Furstenberg and Weiss in \cite{FurWei78}. 
Here we show that any mildly mixing transformation is lightly mixing on a dense algebra. 
This implies that any transformation with a mildly mixing factor is lightly mixing on a dense algebra. 
In section \ref{exs}, we show that this is not necessary by exhibiting an example of the cartesian product 
of two rigid transformations which is lightly mixing on a dense algebra.  Below, we give a necessary 
and sufficient condition for a transformation to be lightly mixing on a dense algebra.  
A proof of sufficiency is given, but necessity is more straightforward, so it is left to the reader. 

\begin{theorem}
\label{thm3}
Suppose that given $\delta > 0$ there exist disjoint sets $C$ and $D$ satisfying 
$0 < \mu(C), \mu(D) < \delta$ such that $\liminf_{n\to \infty} \mu(T^n C\cap D) > 0$.  
Then there exists a dense algebra $\mathcal{G}$ such that for $G,H \in \mathcal{G}$, 
\[
\liminf_{n\to \infty} \mu (T^n G\cap H) > 0 .
\]
\end{theorem}

\begin{proof} 
For each $h\in \natural$, choose disjoint sets $\hat{C_h}$ and $\hat{D_h}$ with measure less than 
$1 / h$ such that 
\[
\liminf_{n\to \infty} \mu(\hat{D_h} \cap T^n \hat{C_h}) \geq \epsilon_h > 0 .
\] 
Choose a sequence $h_m$ such that 
\[
\sum_{n=m+1}^{\infty} \mu( \hat{C}_{h_n} \cup \hat{D}_{h_n} ) < \frac{\epsilon_{h_m}}{4} . 
\]
Let 
\[
C_m = \hat{C}_{h_m} \setminus \bigcup_{n=m+1}^{\infty} ( \hat{C}_{h_n} \cup \hat{D}_{h_n} ) 
\] 
and 
\[
D_m = \hat{D}_{h_m} \setminus \bigcup_{n=m+1}^{\infty} ( \hat{C}_{h_n} \cup \hat{D}_{h_n} ) . 
\] 
Thus 
\[
\liminf_{n\to \infty} \mu (T^n C_m \cap D_m ) \geq \frac{\epsilon_{h_m}}{2} . 
\] 
Hence $\{ C_m, D_m : m\in \natural \}$ is a disjoint collection of sets such that 
$\lim_{m\to \infty} \mu(C_m) = \lim_{m\to \infty} \mu(D_m) = 0$ and 
$\liminf_{n\to \infty} \mu(T^n C_m \cap D_m) > 0$ for each $m\in \natural$.  Let 
\[
\mathcal{E} = \{ E\in \mathcal{B}: E \cap (C_m \cup D_m) \in \{ \emptyset, C_m,D_m,C_m\cup D_m \} \} . 
\]
Let two sets $E$ and $F$ be equivalent (written $E \sim F$) if 
$E \cap (C_m \cup D_m) = F\cap (C_m \cup D_m)$ for all but finitely many $m$. 

Let $F_i, i=1,2,\ldots$ be a sequence of sets such that for each positive integer $n$ and 
$f_i \in \{ \emptyset, C_m, D_m, C_m \cup D_m \}$ for $1\leq i\leq n$, the set 
$\{ m : F_i \cap (C_m \cup D_m) = f_i\ \mbox{for}\ 1\leq i\leq n \}$ is infinite.  
Such a sequence may be constructed inductively on $i$.  Let 
\[
\mathcal{F} = \{ F\in \mathcal{B}:\ \mbox{there exists}\ i\ \mbox{such that either}\ 
F\sim F_i\ \mbox{or}\ F\sim F_i^c \} . 
\]
Let $\mathcal{G}$ be the algebra generated by $\mathcal{F}$. 

Let $k$ be a positive integer and let 
$F_i^{\prime}, F_i^{\prime\prime} \in \mathcal{F}$ for $1\leq i\leq k$ be such that 
$F_i^{\prime} \sim F_i$ or $F_i^{\prime} \sim F_i^c$, and $F_i^{\prime\prime} \sim F_i$ 
or $F_i^{\prime\prime} \sim F_i^c$.  
We wish to show that $\liminf_{n\to \infty} \mu(T^n G\cap H) > 0$ where 
$G=\cap_{i=1}^{k} F_i^{\prime}$ and $H=\cap_{i=1}^{k} F_i^{\prime\prime}$.  
It is sufficient to prove that there exists $m\in \natural$ such that 
$C_m \subset G$ and $D_m \subset H$.  
Let 
\[
S = \{ i: F_i^{\prime} \sim F_i^{\prime\prime}, 1\leq i\leq k \} . 
\]
Then the sets $\{ m: F_i^{\prime} \cap (C_m \cup D_m) = C_m \cup D_m$ 
for $i\in S$ and $F_i^{\prime} \cap (C_m\cup D_m) = C_m$ for $i\notin S \}$ 
and 
$\{ m: F_i^{\prime\prime} \cap (C_m \cup D_m) = C_m \cup D_m$ 
for $i\in S$ and $F_i^{\prime\prime} \cap (C_m\cup D_m) = D_m$ for $i\notin S \}$ 
agree for all but a finite number of positive integers.  
Since both sets are infinite, there exists $m\in \natural$ such that 
$C_m \subset \cap_{i=1}^{k} F_i^{\prime}$ and 
$D_m \subset \cap_{i=1}^{k} F_i^{\prime\prime}$ . 
\end{proof}

\begin{corollary}
\label{factor_light}
If a transformation has a factor which is lightly mixing on a dense algebra, then the 
transformation is itself lightly mixing on a dense algebra. 
\end{corollary}
\begin{proof}
Suppose that a transformation has a factor which is lightly mixing on a dense algebra. 
Then there exist disjoint sets of arbitrarily small measure which lightly mix with each other.  
Hence, by Theorem \ref{thm3}, the transformation itself is lightly mixing on a dense algebra. 
\end{proof}

\begin{corollary}
\label{sodc_lmda}
If a transformation is sweeping out on a dense collection then the transformation 
is lightly mixing on a dense algebra. 
\end{corollary}
\begin{proof}
If a transformation is sweeping out on a dense collection then there exist disjoint sets 
of arbitrarily small measure which lightly mix with each other.  So, again, 
by Theorem \ref{thm3}, the transformation is lightly mixing on a dense algebra. 
\end{proof} 

\begin{theorem}
\label{thm4}
If a transformation $T$ has a mildly mixing factor, then $T$ is lightly mixing 
on a dense algebra. 
\end{theorem}
\begin{proof}
By Corollary \ref{factor_light}, it is sufficient to prove this theorem where $T$ is mildly mixing. 
Let $h$ be a positive integer and let $\{ B, TB, \ldots , T^h B, E \}$ 
be a Rohlin tower such that $E = (\bigcup_{i=0}^h T^i B)^c$ satisfies 
$\mu (E) < \frac{\mu (B)}{4(h+1)}$.  
Let $\hat{E} = \{ x \in E : \exists \ i \in \natural\ \mbox{such that}\ T^i x \in B \}$.  
For $x\in \hat{E}$, let 
$i_x = \min{ \{ i \in \natural : \ T^i(x) \in B \} }$. 
If $\mu (E) > 0$, let 
\[
C_h = \bigcup_{x\in \hat{E}} \{ T^i x : 0 \leq i \leq i_x + h \} , 
\]
otherwise, choose $\hat{A} \subset B$ such that 
$0 < \mu (\hat{A}) < \frac{\mu(B)}{4(h+1)}$ and let 
\[
C_h = \bigcup_{i=0}^h T^i \hat{A} . 
\]
Also, let 
\[
D_h = (B\setminus C_h) \cup (T^h B \setminus C_h) . 
\]
Thus $C_h$ and $D_h$ are disjoint sets, both with measure less than 
$2/{(h+1)}$.  Since $T$ is mildly mixing, 
\[
\epsilon := \liminf_{n\to \infty} \mu (T^n C_h \cap C_h^c ) > 0 . 
\]
For $n$ sufficiently large choose $j_n$ with $0\leq j_n \leq h$ such that 
\[
\mu( T^n C_h \cap (T^{j_n} B\cap C_h^c )) > \frac{\epsilon}{h + 2} . 
\]
For each $x\in T^n C_h \cap (T^{j_n} B\cap C_h^c )$, either 
$T^{-j_n} x \in T^n C_h$ or $T^{h-j_n}x \in T^n C_h$.  
Either way, there exists a measure preserving map $\phi$ such that for each 
$x\in T^n \hat{C}_h \cap (T^{j_n} B \cap \hat{C}_h^c )$, $\phi (x) \in \hat{D}_h \cap T^n \hat{C}_h$.  
Therefore, 
\[
\liminf_{n\to \infty} \mu (\hat{D}_h \cap T^n \hat{C}_h ) \geq \frac{\epsilon}{h + 2} . 
\]
\end{proof} 

\section{Sweeping Out}
A measure preserving transformation $T$ is sweeping out on a set $A$ if given any 
sequence $k_1, k_2, \ldots$ of distinct integers, 
\[
\mu (\bigcup_{i=1}^{\infty} T^{k_i} A) = 1 . 
\]
A transformation $T$ is sweeping out on a collection $\mathcal{D}$ of sets if $T$ is 
sweeping out on each set $A\in \mathcal{D}$.  In the previous section, we proved 
that sweeping out on a dense collection implies lightly mixing on a dense algebra.  
In order to see that the converse is not true, we give a general condition on the limit joinings 
of a transformation which implies that it does not sweep out on any dense collection 
of sets.  Then we point out in section \ref{exs} that Chacon's transformation, which is 
mildly mixing, satisfies this condition.  Since any mildly mixing transformation is 
lightly mixing on a dense algebra, then Chacon's transformation is lightly mixing 
on a dense algebra, but does not sweep out on a dense collection. 

Before we do this, we will give an equivalent condition for a transformation 
to be sweeping out on a dense algebra.  Then we use this condition to show 
that any transformation with a factor that sweeps out on a dense algebra, is itself 
sweeping out on a dense algebra. 

\begin{theorem}
\label{thm5}
A transformation $T$ is sweeping out on a dense algebra if and only if there 
exist disjoint sets $C_m$ for $m\in \natural$ such that for each $m$, $T$ 
sweeps out on $C_m$. 
\end{theorem}
\begin{proof}
The proof is similar to the proof of Theorem \ref{thm3}, except we let 
\[
\mathcal{E} = \{ E\in \mathcal{B} : E\cap C_m \in \{ \emptyset, C_m\}\} . 
\]
Let two sets $E$ and $F$ be equivalent (written $E\sim F$) if $E\cap C_m = F\cap C_m$ 
for all but finitely many $m$. 

For $i=1,2,\ldots$ , define 
\[
F_i = \bigcup_{m=0}^{\infty} \bigcup_{j=1}^{2^i} C_{m2^{i+1}+j} . 
\]
Let 
\[
\mathcal{F} = \{ F\in \mathcal{B}:\mbox{ there exists $i$ such that either } F\sim F_i\mbox{ or } F\sim F_i^c \} . 
\]
Let $\mathcal{G}$ be the algebra generated by $\mathcal{F}$. 

To see that $T$ sweeps out on $\mathcal{G}$ let $k$ be a positive integer and let 
$F_i^{\prime} \in \mathcal{F}$ for $1\leq i\leq k$ be such that 
$F_i^{\prime} \sim F_i$ or $F_i^{\prime} \sim F_i^c$.  From the definition of $F_i$, 
$G = \bigcap_{i=1}^k F_i^{\prime}$ contains infinitely many $C_m$.  Therefore, $T$ is sweeping out on $G$. 
\end{proof}

\begin{corollary}
\label{factor_cor}
If a transformation has a factor which is sweeping out on a dense algebra, then the transformation itself 
is sweeping out on a dense algebra.  Moreover, if a transformation has a factor which satisfies 
LMDC, LMDA, SODC or SODA, then the transformation satisfies the same property. 
\end{corollary}
\begin{proof}
Property LMDA is handled in Corollary \ref{factor_light}.  If a transformation has a factor which is SODA, 
then the transformation satisfies the equivalent condition of Theorem \ref{thm5}, 
and hence the transformation is SODA.  

Suppose a transformation has a factor which satisifies LMDC or SODC.  
Let $\mathcal{D}$ be the dense collection 
for which the factor is LMDC or SODC.  Then the transformation is LMDC or SODC, 
respectively, on the dense collection 
\[
\{ A\in \mathcal{B} : \mbox{ there exists $D\in \mathcal{D}$ such that } D\subset A \} . 
\]
\end{proof}

\begin{theorem}
\label{thm6}
Let $T$ be measure preserving on $(X, \mathcal{B}, \mu )$.  
If there exists a positive integer $N$ and real numbers $a_i$ for 
$i = 0,\ldots ,N$ such that $T$ has the limit joining $\nu$ given by 
\[
\nu (A\times B) = \sum_{i=0}^N a_i \mu (T^{-i} A\cap B) , 
\]
then $T$ does not sweep out on a dense collection of sets. 
\end{theorem}
\begin{proof}
Suppose $k_n$ is a sequence such that 
\begin{eqnarray}
\lim_{n\to \infty} \mu (T^{k_n} A\cap B) = \nu (A\times B) = 
\sum_{i=0}^{N} a_i \mu (T^{-i} A\cap B) 
\end{eqnarray} 
for all sets $A$ and $B$ in $\mathcal{B}$.  Let $A$ be any set such that 
$0 < \mu(A) < \frac{1}{2(N+1)}$.  
Define $B = (\bigcup_{i=0}^N T^{-i} A)^c$. 
Thus $\lim_{n\to \infty} \mu (T^{k_n} A\cap B) = \nu (A\times B) = 0$. 
\end{proof} 

\section{Examples}
\label{exs}
In this section, we give four examples which are intended to compliment and 
contrast with the theorems in the previous sections. 

\subsection{Example 1}
Chacon's transformation may be defined inductively using blocks $B_n$.  
Let $B_1 = (0)$ and let $B_{n+1} = B_n B_n 1 B_n$.  
If $k_n$ is the length of $B_n$, then it is not difficult to verify 
for all sets $A$ and $B$, 
\begin{eqnarray}
\label{joineq1}
\nu (A\times B) &=& \lim_{n\to \infty} \mu (T^{k_n} A\cap B) \\ 
&=& \frac12 \mu (A\cap B) + \frac12 \mu (T^{-1} A\cap B) . \label{joineq2}
\end{eqnarray}
Therefore, by Theorems \ref{thm4} and \ref{thm6}, Chacon's transformation is lightly mixing 
on a dense algebra, but not sweeping out on a dense collection. 
Note that if equations \ref{joineq1} and \ref{joineq2} above are true for a dense collection 
then it is true for all sets in the sigma algebra. 


\subsection{Example 2}
\label{ex2} 
Let $R$ be a lightly mixing transformation on $(X, \mathcal{B}, \mu)$ and let $S$ be a weak 
mixing transformation on $(Y, \mathcal{C}, \nu)$, which is not lightly mixing. 
Since the cartesian product $R\times S$ has a factor which sweeps out, then by 
Corollary \ref{factor_cor}, $R\times S$ sweeps out on a dense algebra.  
But, since $S$ is not lightly mixing, then $R\times S$ is not lightly mixing. 

A stronger sweeping out condition was introduced by Friedman in \cite{Fri83} called 
uniform sweeping out.  A transformation $T$ is uniformly sweeping out on a set $A$ if given 
$\epsilon > 0$, there exists $N$ such that for distinct integers $k_1, k_2,\ldots ,k_N$, 
we have 
\[
\mu (\bigcup_{i=1}^N T^{k_i} A) > 1 - \epsilon . 
\] 
A transformation is uniformly sweeping out (USO) if it is uniformly sweeping out 
on each set of positive measure.  It was shown in \cite{Fri83} that if a transformation is 
mixing then it is uniformly sweeping out. 

Just as with other sweeping out conditions we have considered to this point, 
a transformation, which has a factor which is uniformly sweeping out on a dense 
collection (USODC) or dense algebra (USODA), is itself uniformly sweeping out 
on a dense collection or dense algebra, respectively.  Thus, 
by forming products $R\times S$ where $R$ is a mixing transformation 
and $S$ is various other transformations, we can construct the following examples: 
\begin{enumerate}
\item a partially mixing transformation which is USODA, but not USO; 
\item a mildly mixing transformation which is USODA, but not LM; 
\item a weakly mixing transformation which is USODA, but not MM. 
\end{enumerate} 

\subsection{Example 3} 
\label{ex3} 
Below we describe how to construct a transformation which sweeps out on a dense collection, 
but which does not have a mildly mixing factor. 
Let $R_1$ and $R_2$ be rigid transformations with mixing sequences 
$M_1 = \{ m_{1,1}, m_{1,2}, \ldots \}$ and $M_2 = \{ m_{2,1}, m_{2,2}, \ldots \}$, 
respectively, such that $M_1 \cup M_2 = \natural$.  
Explicit transformations may be constructed using cutting and stacking: 
initially, build $R_1$ and $R_2$ as mixing transformations; then introduce a rigid time 
into $R_1$ while keeping $R_2$ mixing.  Then resume a mixing construction on $R_1$ 
and introduce a rigid time into $R_2$.  This routine may be carried out ad infinitum. 
The result is that $R_1 \times R_2$ will sweep out on the dense collection 
\[
\{ A : \exists \mbox{ $E$ and $F$ of positive measure such that } 
(E\times Y) \cup (X\times F) \subset A \} . 
\]

\subsection{Sweeping Out on a Residual Class} 
Each of examples \ref{ex2} and \ref{ex3} is sweeping out on a residual class of sets. 
Let us define the class for example \ref{ex2}.  Example \ref{ex3} may be defined 
in a similar manner.  Let 
\[
\mathcal{D}_n = \{ A : \exists E > 0\mbox{ and } A^{\prime} \mbox{ such that } 
E\times Y \subset A^{\prime}, \mu(A\triangle A^{\prime}) < \frac{1}{n} \mu (E) \} . 
\]
For each $n \in \natural$, the set $\mathcal{D}_n$ is open and dense.  
Hence $\mathcal{D} = \bigcap_{n=1}^{\infty} \mathcal{D}_n$ is residual.  
To prove $R\times S$ is sweeping out on $\mathcal{D}$, let $A \in \mathcal{D}$ 
and $k_1, k_2, \ldots$ be an infinite sequence of integers.  Given $\epsilon > 0$, 
we may choose $E \in \mathcal{B}$ of positive measure such that 
$\mu \times \mu (A\cap E\times Y) > (1-\epsilon )\mu (E)$.  
Thus the set $F = \{ x\in E : \mu \{ y : (x,y) \in A \} > 1 - \epsilon \}$ 
has positive measure.  Since $R$ is sweeping out, then 
\[
\mu \times \mu (\bigcup_{i=1}^{\infty} (R\times S)^{k_i} A ) > 
(1-\epsilon) \mu (\bigcup_{i=1}^{\infty} R^{k_i} F ) = 1-\epsilon . 
\]
Since $\epsilon > 0$ is arbitrary, $R \times S$ is sweeping out on $\mathcal{D}$. 

\subsection{Example 4}
Here we give an example of a lightly mixing transformation $T$ such that 
there exists a set $A$ of positive measure such that given $\delta > 0$ 
there exists a set $B$ of positive measure such that 
$\liminf_{n\to \infty} \mu (T^n A\cap B) < \delta \mu (B)$.  This example is 
important since it says that the lightly mixing condition does not imply a partial 
mixing type condition if we fix one of the sets in advance. 
In \cite{Ada93}, a lightly mixing transformation $T_r$ is constructed for each 
$r = 2, 3, \ldots$ which is $1 - ({1} / {r} )$ rigid.  
Choose a sequence $r_m$ such that 
\[
\prod_{m=1}^{\infty} ( 1 - \frac{1}{\sqrt{r_m}} ) > 0 . 
\]
Define 
\[
T = T_{r_1} \times T_{r_2} \times \ldots 
\]
and let $\overline{\mu} = \mu \times \mu \times \ldots$ 
be the invariant product measure for $T$.  
The transformation $T$ is lightly mixing, since the infinite cartesian product 
of lightly mixing transformations is lightly mixing \cite{King88}.  
Let $A = A_1 \times A_2 \times \ldots$ where 
\[
\mu (A_m) = 1 - \frac{1}{\sqrt{r_m}} . 
\]
For each $m \in \natural$, let 
\[
B_m = X \times \ldots \times X \times A_m^{c} \times X \times \ldots 
\]
Therefore, 
\begin{align*}
\liminf_{n\to \infty} \overline{\mu} (T^n A\cap B_m ) &\leq \frac{1}{r_m} \\ 
&= \frac{1}{r_m} \sqrt{r_m} \mu (B_m) = \frac{1}{\sqrt{r_m}} \mu (B_m) . 
\end{align*}

\subsection{Remark 1}
While the previous example shows that there is no type 
of partially mixing condition, if one merely fixes the set $A$ in advance.  
However, we do not know if there is a lower bound on the $\liminf$ if one fixes 
the set $A$ and places a lower bound on the measure of the set $B$.  
Also, we would find it interesting if there were general conditions which would 
allow one to extend lightly mixing from a dense collection of sets to the full 
sigma algebra. 

\subsection{Remark 2} 
In this remark, we show that 
any mildly mixing (hence weakly mixing) transformation $T$ has a realization 
that is lightly mixing on intervals.  As a consequence, this realization is topologically mixing. 
Theorem \ref{thm4} establishes that there exists 
a dense algebra $\mathcal{C}$ such that $T$ is lightly mixing on $\mathcal{C}$.  
Let $\lambda$ be Lebesgue measure on $[0,1)$. 
There exists an invertible measure space isomorphism 
$\phi : [0,1) \to X$ between $([0,1), \lambda)$ and $(X, \mu)$ 
such that for any interval $I\subset [0,1)$, 
the set $\phi (I)$ contains an element from $\mathcal{C}$.  
Hence, the transformation $\phi^{-1} \circ T \circ \phi$ is 
mildly mixing and lightly mixing on intervals.  
If $T$ is chosen as a non-lightly mixing transformation 
(i.e., Chacon-3 transformation), then this provides 
a general method for extending the examples given in \cite{MRSZ}. 

\subsection{Remark 3}
Both Theorem \ref{thm2} and Theorem \ref{thm4} may be extended to any group 
action of measure preserving transformations which have a weak type 
of Rohlin property.  For example, Theorem \ref{thm2} and \ref{thm4} hold 
for amenable groups, by using the characterization of Rohlin's Lemma 
given in \cite{OW87}.  Any topological group $G$ with measure preserving action $T$ 
is lightly mixing on a collection $\mathcal{C}$ if for all sets $A$ and $B$ in $\mathcal{C}$ 
with positive measure, 
\[
\liminf_{g \to \infty} \mu (T^g A \cap B) > 0 , 
\]
where $g \to \infty$ means eventually leave any compact set in $G$.  
The mildly mixing condition may be replaced by 
\[
\liminf_{g \to \infty} \mu (T^g A \cap A^c) > 0 , 
\] 
for all sets $A$ of positive measure. 

\subsection{Remark 4} 
We do not handle the case of general nonsingular actions.  
Although there is a version of Rohlin's Lemma in this case 
\cite{Kopf79, OW87}, it is not immediately clear that the sets may be chosen 
so that the $\liminf$ is positive. 

\subsection{Remark 5} 
This paper was adapted from a preprint written in 1998 by the same author 
entitled {\it Lightly mixing on dense collections}. 


\end{document}